\documentclass[12pt]{amsart}
\usepackage{amssymb}
\usepackage{amsmath}
\usepackage{amscd}

\theoremstyle{plain} 
\newtheorem{theorem}{Theorem}[section]
\newtheorem{corollary}[theorem]{Corollary}
\newtheorem{lemma}[theorem]{Lemma}
\newtheorem{proposition}[theorem]{Proposition}

\theoremstyle{definition}

\newtheorem{example}[theorem]{Example}

\theoremstyle{remark}
\newtheorem{remark}[theorem]{Remark}

\numberwithin{equation}{section}

\def\N{{\mathbb N}}

\def\cA{{\mathcal A}}
\def\cB{{\mathcal B}}

\def\cH{{\mathcal H}}
\def\cL{{\mathcal L}}

\def\cS{{\mathcal S}}
\def\cT{{\mathcal T}}

\def\e{{\boldsymbol e}}

\def\0{{\boldsymbol 0}}
\def\1{{\boldsymbol 1}}

\def\aalpha{{\boldsymbol{\alpha}}}
\def\bbeta{{\boldsymbol{\beta}}}
\def\ggamma{{\boldsymbol{\gamma}}}

\def\pdeg{{\rm pdeg}}

\def\Hom{{\rm Hom}}

\def\Im{{\rm Im}}
\def\Ker{{\rm Ker}}
\def\Coker{{\rm Coker}}

\def\rank{{\rm rank}}

\title{The Freeness and minimal free resolutions
of Modules of Differential Operators of a Generic Hyperplane
Arrangement
}

\author{Norihiro Nakashima}
\address{Department of Mathematics,
Graduate School of Science,
Hokkaido University,
Sapporo, 060-0810, Japan
}
\email{naka{\_}n@math.sci.hokudai.ac.jp}

\author{Go Okuyama}
\address{Hokkaido Institute of Technology,
Sapporo, 006-8585, Japan
}
\email{gokuyama@hit.ac.jp}

\author{Mutsumi Saito}
\address{Department of Mathematics,
Graduate School of Science,
Hokkaido University,
Sapporo, 060-0810, Japan
}
\email{saito@math.sci.hokudai.ac.jp}

\begin{document}

\begin{abstract}
Let ${\mathcal A}$ be a generic hyperplane arrangement
composed of $r$ hyperplanes in an $n$-dimensional vector space,
and $S$ the polynomial ring in $n$ variables.
We consider the $S$-submodule $D^{(m)}({\mathcal A})$ of the $n$th Weyl algebra
of  homogeneous differential operators
of order $m$ preserving the defining ideal of ${\mathcal A}$.

We prove that
if $n\geq 3, r>n, m>r-n+1$, then $D^{(m)}({\mathcal A})$ is free
(Holm's conjecture).
Combining this with some results by Holm,
we see that
$D^{(m)}({\mathcal A})$ is free unless
$n\geq 3, r>n, m<r-n+1$.
In the remaining case,
we construct a minimal free resolution of
$D^{(m)}({\mathcal A})$
by generalizing Yuzvinsky's construction for $m=1$.
In addition, we construct a minimal free resolution
of the transpose of the $m$-jet module, which generalizes
a result by Rose and Terao for $m=1$.


\end{abstract}


\maketitle


\section{Introduction}
\label{INTRO}

In the study of a hyperplane arrangement,
its derivation module plays a central character;
in particular, its freeness over the polynomial ring
attracts a great interest
 (see, e.g., Orlik-Terao \cite{Orlik-Terao}).
Generalizing the study of the derivation module
for a hyperplane arrangement
to that of the modules of differential operators of higher order
was initiated by 
Holm \cite{Holm-thesis}, \cite{Holm04}.
In particular, he studied  the case of
generic hyperplane arrangements in detail.

Let $K$ denote a field of characteristic zero, and
$\cA$ a generic hyperplane arrangement in $K^n$ composed of
$r$ hyperplanes.
Let $S$ be the polynomial ring $K[x_1,\ldots, x_n]$,
and $D^{(m)}(\cA)$ the $S$-module of homogeneous differential operators of
order $m$ of the hyperplane arrangement $\cA$.

Among others, in \cite{Holm04},
Holm gave a finite generating set of the $S$-module $D^{(m)}(\cA)$.
As to the freeness of $D^{(m)}(\cA)$,
Holm \cite{Holm-thesis} (cf. \cite{Snellman})
proved the following:

\begin{itemize}
\item
If $n=2$, then $D^{(m)}(\cA)$ is free for any $m$.
\item
If $n\geq 3, r>n, m<r-n+1$, then $D^{(m)}(\cA)$ is not free.
\item
If $n\geq 3, r>n, m=r-n+1$, then $D^{(m)}(\cA)$ is free.
\end{itemize}
Holm also conjectured that
if $n\geq 3, r>n, m>r-n+1$, then $D^{(m)}(\cA)$ is free.

Snellman \cite{Snellman}
computed the Hilbert series of $D^{(m)}(\cA)$,
which supported Holm's conjecture
when $n\geq 3, r>n, m>r-n+1$,
and he
conjectured the
Poicar\'e-Betti series of $D^{(m)}(\cA)$
when
$n\geq 3, r>n, m<r-n+1$.

In the derivation module case,
when $n\geq 3, r>n, m<r-n+1$ with $m=1$,
Rose-Terao \cite{Rose-Terao} and Yuzvinsky \cite{Yuzvinsky}
independently gave a minimal free resolution of  $D^{(1)}(\cA)$.
In the course of the proof,
Rose-Terao \cite{Rose-Terao} 
gave
 minimal free resolutions of
all modules of logarithmic differential forms with poles along $\cA$.
They also gave a minimal free resolution of
$S/J$, where $J$ is the Jacobian ideal of a polynomial defining $\cA$.
Yuzvinsky's construction \cite{Yuzvinsky}
is more straightforward and combinatorial than \cite{Rose-Terao}.

In this paper, we prove Holm's conjecture,
namely,
we prove that
if $n\geq 3, r>n, m>r-n+1$, then $D^{(m)}(\cA)$ is free.
Hence, for a generic hyperplane arrangement $\cA$,
$D^{(m)}(\cA)$ is free unless
$n\geq 3, r>n, m<r-n+1$.
In the remaining case $n\geq 3, r>n, m<r-n+1$,
we construct a minimal free resolution of
$D^{(m)}(\cA)$
by generalizing \cite{Yuzvinsky} and a minimal free resolution
of the transpose of the $m$-jet module generalizing
that of $S/J$ given by \cite{Rose-Terao}.

After we fix notation on differential operators for a hyperplane arrangement
in \S 2, we recall the Saito-Holm criterion in \S 3.
It was proved by Holm, and it is a criterion for a subset of $D^{(m)}(\cA)$
to form a basis,
which generalizes the Saito criterion in the case of $m=1$.

From \S 4 on, we assume that $r\geq n$ and the hyperplane arrangement $\cA$
is generic.
In \S 4, we recall the finite generating set of $D^{(m)}(\cA)$
given by Holm \cite{Holm04}.
Then we recall the case $n=2$ in \S 5 and the case $m=r-n+1$
in \S 6 for completeness.
In \S 7, we consider the case $m\geq r-n+1$ and prove Holm's conjecture
(Theorem \ref{Holm'sConjecture}).

From \S 8 on, we consider the case $m<r-n+1$.
In \S 8, we give a minimal generating set of $D^{(m)}(\cA)$
(Theorem \ref{theorem10.3}).
In \S 9, we generalize \cite{Yuzvinsky} to construct a minimal free
resolution of $D^{(m)}(\cA)$ (Theorem \ref{thm:MinFreeRes}).
In \S \ref{Jacobian}, we generalize the minimal free resolution of $S/J$
given in \cite{Rose-Terao} (Theorem \ref{Rose-Terao:Theorem4.5.3}).
In \S 11, we prove that
the $S$-module considered in \S \ref{Jacobian}
is the transpose of the $m$-jet module
$\Omega^{[1,m]}(S/SQ)$ (Theorem \ref{thm:TransposeOfJet}),
where $Q$ is a polynomial defining $\cA$.

\section{The Modules of Differential Operators for a Hyperplane Arrangement}

Throughout this paper,
let $K$ denote a field of characteristic zero,
$\cA$ a central hyperplane arrangement in $K^n$ composed of
$r$ hyperplanes, 
and $S$ the polynomial ring $K[x_1,\ldots, x_n]$.
We assume that $n\geq 2$.

For a hyperplane $H\in \cA$,
we fix a linear form $p_H\in S$ defining $H$.
Set
\begin{equation}
Q:=Q_\cA:=\prod_{H\in\cA}p_H.
\end{equation}

Let $D(S)=S\langle \partial_1,\ldots, \partial_n\rangle$ 
denote the $n$th Weyl algebra,
where $\partial_j= \frac{\partial}{\partial x_j}$.
For a nonzero differential operator
$P=\sum_{\aalpha\in \N^n}f_\aalpha(x)\partial^\aalpha\in D(S)$,
the maximum of $|\aalpha|$ with $f_\aalpha\neq 0$
is called the {\it order} of $P$, where
$$
\partial^\aalpha=\partial_1^{\alpha_1}\cdots \partial_n^{\alpha_n},
\qquad
|\aalpha|=\alpha_1+\cdots +\alpha_n
$$
for $\aalpha=(\alpha_1, \ldots, \alpha_n)$.
If $P$ has no nonzero $f_\aalpha$ with $|\aalpha|\neq m$,
it is said to be {\it homogeneous of order $m$}.
We denote
by $D^{(m)}(S)$ the $S$-submodule of $D(S)$
of differential operators 
homogeneous of order $m$.

We denote by $\ast$ the action of $D(S)$ on $S$.
For an ideal $I$ of $S$, 
\begin{equation}
\label{def:D(I)}
D(I):=
\{ \theta\in D(S)\, |\, \theta\ast I \subseteq I\}
\end{equation}
is called the {\it idealizer\/} of $I$.

We set
\begin{equation}
\label{def:D(A)}
D(\cA):=D(\langle Q\rangle).
\end{equation}
Holm \cite[Theorem 2.4]{Holm04} proved
\begin{equation}
\label{eqn:2.4}
D(\cA)=\bigcap_{H\in\cA}D(\langle p_H\rangle).
\end{equation}
We denote
by $D^{(m)}(\cA)$ the $S$-submodule of $D(\cA)$
of differential operators
homogeneous of order $m$.
Then Holm \cite[Proposition 4.3]{Holm04} proved
$$
D(\cA)=\bigoplus_{m=0}^\infty D^{(m)}(\cA).
$$
A differential operator homogeneous of order $1$
is nothing but a derivation.
Hence $D^{(1)}(\cA)$ is the module of logarithmic derivations along $\cA$.

The polynomial ring $S=\bigoplus_{p=0}^\infty S_p$
is a graded algebra, where $S_p$ is the $K$-vector subspace
spanned by the monomials of degree $p$.
The $n$th Weyl algebra $D(S)$ is a graded $S$-module with
$\deg(x^\aalpha\partial^\bbeta)=|\aalpha|-|\bbeta|$.
Each $D^{(m)}(\cA)$ is a graded $S$-submodule of $D(S)$.
An element 
$P=\sum_{\aalpha\in \N^n}f_\aalpha(x)\partial^\aalpha\in D^{(m)}(\cA)$
is said to be {\it homogeneous of polynomial degree $p$},
and denoted by $\pdeg P=p$,
if $f_\aalpha\in S_p$ for all $\aalpha$ with nonzero $f_\aalpha$.

\section{Saito-Holm criterion}

To prove that $D^{(1)}(\cA)$ is a free $S$-module,
the Saito criterion (\cite[Theorem 1.8 (ii)]{Kyoji}, 
see also \cite[Theorem 4.19]{Orlik-Terao})
is very useful.
Holm \cite{Holm-thesis} generalized the Saito criterion
to the one for $D^{(m)}(\cA)$.
In this section, we briefly review Holm's generalization.

Set
$$
s_m:=\binom{n+m-1}{m},
\qquad
t_m:=\binom{n+m-2}{m-1}.
$$
Let
$$
\{ x^{\aalpha^{(1)}}, x^{\aalpha^{(2)}},\ldots, x^{\aalpha^{(s_m)}}\} 
$$
be the set of monomials of degree $m$.
For operators $\theta_1, \ldots, \theta_{s_m}$,
define an $s_m\times s_m$ coefficient matrix 
$M_m(\theta_1, \ldots, \theta_{s_m})$ by
$$
M_m(\theta_1, \ldots, \theta_{s_m}):=
\begin{bmatrix}
\theta_1\ast \frac{x^{\aalpha^{(1)}}}{\aalpha^{(1)} !} 
& \cdots &
\theta_{s_m}\ast \frac{x^{\aalpha^{(1)}}}{\aalpha^{(1)} !}\\
\vdots & \ddots & \vdots \\
\theta_1\ast \frac{x^{\aalpha^{(s_m)}}}{\aalpha^{(s_m)} !} 
& \cdots &
\theta_{s_m}\ast \frac{x^{\aalpha^{(s_m)}}}{\aalpha^{(s_m)} !}
\end{bmatrix},
$$
where $\aalpha !=(\alpha_1!)(\alpha_2!)\cdots (\alpha_n!)$
for $\aalpha=(\alpha_1,\alpha_2, \ldots, \alpha_n)$.

The proofs of the following
two propositions go similarly to those of \cite[Proposition 4.12]{Orlik-Terao}
and \cite[Proposition 4.18]{Orlik-Terao}.

\begin{proposition}[I$\!$I$\!$I Proposition 5.2 in \cite{Holm-thesis} 
(cf. Proposition 4.12 in \cite{Orlik-Terao})]
\label{Proposition4.12}
If $\theta_1, \ldots, \theta_{s_m}\in D^{(m)}(\cA)$,
then
$$
\det M_m(\theta_1, \ldots, \theta_{s_m})
\in \langle Q^{t_m}\rangle.
$$
\end{proposition}

\begin{proposition}[I$\!$I$\!$I Proposition 5.7 in \cite{Holm-thesis} 
(cf. Proposition 4.18 in \cite{Orlik-Terao})]
\label{Proposition4.18}
Suppose that $D^{(m)}(\cA)$ is a free $S$-module.
Then the rank of $D^{(m)}(\cA)$ is $s_m$.
\end{proposition}

The following is a generalization of the Saito criterion.
This was proved by Holm \cite[I$\!$I$\!$I Theorem 5.8]{Holm-thesis}.

\begin{theorem}[Saito-Holm criterion]
\label{Saito'sCriterion}
Given $\theta_1, \ldots, \theta_{s_m}\in D^{(m)}(\cA)$,
the following two conditions are equivalent:
\begin{enumerate}
\item[{\rm (1)}]
$\det M_m(\theta_1, \ldots, \theta_{s_m})
= c Q_\cA^{t_m}$ for some $c\in K^\times$,
\item[{\rm (2)}]
$\theta_1, \ldots, \theta_{s_m}$ form a basis for $D^{(m)}(\cA)$ over $S$.
\end{enumerate}
\end{theorem}

The following is an easy consequence of Theorem \ref{Saito'sCriterion}.

\begin{theorem}[I$\!$I$\!$I Theorem 5.9 in \cite{Holm-thesis} 
(cf. Theorem 4.23 in \cite{Orlik-Terao})]
\label{Theorem4.23}
Let $\theta_1, \ldots, \theta_{s_m}\in D^{(m)}(\cA)$ be linearly independent
over $S$.
Then $\theta_1, \ldots, \theta_{s_m}$ form a basis for $D^{(m)}(\cA)$ over $S$ if and only if
$$
\sum_{j=1}^{s_m}\pdeg\, \theta_j= r t_m.
$$
\end{theorem}

Suppose that $D^{(m)}(\cA)$ is free over $S$.
We denote by $\exp D^{(m)}(\cA)$
the multi-set of polynomial degrees of a basis for $D^{(m)}(\cA)$.
The expression
$$
\exp D^{(m)}(\cA)= \{ 0^{e_0}, 1^{e_1}, 2^{e_2},\ldots\}
$$
means that $\exp D^{(m)}(\cA)$ has $e_i$ $i$'s ($i=0,1,2,\cdots$).

\begin{proposition}[cf. Proposition 4.26 in \cite{Orlik-Terao}]
\label{Proposition4.26}
Assume that $D^{(m)}(\cA)$ is free over $S$, and suppose that
$$
\exp D^{(m)}(\cA)= \{ 0^{e_0}, 1^{e_1}, 2^{e_2},\ldots\}.
$$
Then
$$
\sum_k e_k= s_m,\qquad
\sum_k ke_k = r t_m.
$$
\end{proposition}

\begin{proof}
Proposition \ref{Proposition4.18} is the first statement,
and Theorem \ref{Theorem4.23} the second.
\end{proof}


\section{Generic arrangements}

{\bf In the rest of this paper, we assume that $r\geq n$ and $\cA$ is generic.}
An arrangement $\cA$ is said to be {\it generic},
if every $n$ hyperplanes of $\cA$ intersect only at the origin.

For a finite set $\cS$,
let $\cS^{(k)}\subseteq 2^{\cS}$ denote the set of 
$\cT\subseteq \cS$ with $\sharp\cT=k$.

Given $\cH\in \cA^{(n-1)}$,
the vector space
$$
\{ \delta\in \sum_{i=1}^n K\partial_i\, |\,
\text{$\delta\ast p_H=0$ for all $H\in\cH$}\}
$$
is one-dimensional;
fix a nonzero element $\delta_\cH$ of this space.
Note that 
\begin{equation}
\delta_\cH\ast p_H=0 \Leftrightarrow H\in \cH,
\end{equation}
since $\cA$ is generic.

For $\cH_1,\ldots,\cH_m\in \cA^{(n-1)}$,
put
\begin{equation}
P_{\{\cH_1,\ldots,\cH_m\}}:= \prod_{H\notin \cap_{i=1}^m\cH_i}p_H.
\end{equation}
Then $P_{\{\cH_1,\ldots,\cH_m\}}\delta_{\cH_1}\cdots\delta_{\cH_m}\in
D^{(m)}(\cA)$ by \eqref{eqn:2.4}.
In particular, for $\cH\in \cA^{(n-1)}$,
$$
P_{\cH}\delta_{\cH}^m\in D^{(m)}(\cA),
$$
where $P_{\cH}:= P_{\{\cH\}}$.
Note that
\begin{equation}
\label{degPH}
\deg P_{\cH}= r-n+1.
\end{equation}

The operator
\begin{equation}
\epsilon_m:=\sum_{|\aalpha|=m}\frac{m!}{\aalpha !}x^\aalpha\partial^\aalpha
\end{equation}
is called {\it the Euler operator of order $m$}.
Then $\epsilon_1$ is the Euler derivation, and
$\epsilon_m=\epsilon_1(\epsilon_1-1)\cdots(\epsilon_1-m+1)$
\cite[Lemma 4.9]{Holm04}.

Holm gave a finite set of generators of $D^{(m)}(\cA)$
as an $S$-module:

\begin{theorem}[Theorem 4.22 in \cite{Holm04}]
\label{Theorem5.28}
$$
D^{(m)}(\cA)=\sum_{\cH_1,\ldots,\cH_m\in \cA^{(n-1)}}
SP_{\{\cH_1,\ldots,\cH_m\}}\delta_{\cH_1}\cdots\delta_{\cH_m}
+S\epsilon_m.
$$
\end{theorem}

The following lemma will be used in Sections 7, 8, and 9.

\begin{lemma}
\label{DualBasis}
\begin{enumerate}
\item[{\rm (1)}] 
The set
$\{ \delta_\cH^{r-n+1}\, |\, \cH\in \cA^{(n-1)}\}$
is a $K$-basis of $\sum_{|\aalpha|=r-n+1}K\partial^\aalpha$.
\item[{\rm (2)}] 
The set
$\{ P_\cH\, |\, \cH\in \cA^{(n-1)}\}$
is a $K$-basis of $\sum_{|\aalpha|=r-n+1}Kx^\aalpha=S_{r-n+1}$.
\end{enumerate}
\end{lemma}

\begin{proof}
The dimensions of $\sum_{|\aalpha|=r-n+1}K\partial^\aalpha$ and 
$S_{r-n+1}$ are equal to
$$
s_{r-n+1}=
\binom{r}{r-n+1}=\binom{r}{n-1}=\sharp \cA^{(n-1)}.
$$

Let $\cH, \cH' \in \cA^{(n-1)}$. Then
\begin{equation}
\label{eqn:dual}
\delta_\cH^{r-n+1}\ast P_{\cH'}
=
\delta_\cH^{r-n+1}\ast \prod_{H\notin\cH'}p_H
=
\left\{
\begin{array}{ll}
(r-n+1)!\prod_{H\notin\cH}(\delta_\cH\ast p_H) & \text{if $\cH'=\cH$}\\
0 & \text{otherwise}.
\end{array}
\right.
\end{equation}
The assertions follow,
since $\delta_\cH\ast p_H=0$ if and only if $H\in\cH$.
\end{proof}

\section{The case $n=2$}

In this section, we consider central arrangements with $r\geq 2$ in $K^2$,
which are always generic.
Note that $s_m=m+1$, and $t_m=m$.

Let $\cA=\{ H_1, H_2,\ldots, H_r\}$. 
Put
$p_i:=p_{H_i}$, $P_i:=P_{\{ H_i\}}$,
and $\delta_i:=\delta_{\{ H_i\}}$
for $i=1,2,\ldots, r$.

We may assume that there exist distinct $a_2,\ldots, a_r\in K$ such that
$$
p_1=x_1,\qquad p_i=x_2-a_ix_1\quad (i=2,\ldots, r).
$$
Then
$$
\delta_1=\partial_2,\qquad \delta_i=\partial_1+a_i\partial_2\quad (i=2,\ldots, r),
$$
and
$$
P_i=Q/p_i \quad (i=1,\ldots, r).
$$

\begin{proposition}[Proposition 6.7 I$\!$I$\!$I in \cite{Holm-thesis},
Proposition 4.14 in \cite{Snellman}]
\label{Proposition 6.7Holm-thesis3}
The $S$-module
$D^{(m)}(\cA)$ is free with the following basis:
\begin{enumerate}
\item[{\rm (1)}]
$\{ \epsilon_m, P_1\delta_1^m, \ldots, P_m\delta_m^m\}$
if $m\leq r-2$.
\item[{\rm (2)}]
$\{ P_1\delta_1^m, \ldots, P_r\delta_r^m\}$
if $m= r-1$.
\item[{\rm (3)}]
$\{ P_1\delta_1^m, \ldots, P_r\delta_r^m, Q\eta_{r+1},\ldots, 
Q\eta_{m+1}\}$
if $m\geq r$,
where $\{ \delta_1^m, \ldots, \delta_r^m, \eta_{r+1}, \ldots, \eta_{m+1}\}$ is a $K$-basis
of $\sum_{i=0}^m K\partial_1^i\partial_2^{m-i}$.
\end{enumerate}
\end{proposition}

\begin{corollary}
\label{corollary:exp:n=2}
$$
\exp D^{(m)}(\cA)
=\left\{
\begin{array}{ll}
\{ m^1, (r-1)^m\} & (1\leq m\leq r-2),\\
\{ (r-1)^{m+1}\} & (m=r-1),\\
\{ (r-1)^r, r^{m-r+1}\} & (m\geq r).
\end{array}
\right.
$$
\end{corollary}
\section{The case $m= r-n+1$}
\label{section:m= r-n+1}

In this section, we consider the case $m= r-n+1$.
In this case,
\begin{equation}
s_m=\binom{n+m-1}{m}=\binom{r}{m}=\binom{r}{n-1}.
\end{equation}
Note also that $\deg P_\cH= r-n+1=m$ \eqref{degPH}.

In Sections 7, 8, and 9, we use Lemma \ref{DualBasis}
in the case $m=r-n+1$.
Lemma \ref{DualBasis} reads as follows in this case:

\begin{lemma}
\label{m=r-n+1:DualBasis}
\begin{enumerate}
\item[{\rm (1)}]
The set
$\{ \delta_\cH^m\, |\, \cH\in \cA^{(n-1)}\}$
is a $K$-basis of $\sum_{|\aalpha|=m}K\partial^\aalpha$.
\item[{\rm (2)}]
The set
$\{ P_\cH\, |\, \cH\in \cA^{(n-1)}\}$
is a $K$-basis of $\sum_{|\aalpha|=m}Kx^\aalpha=S_m$.
\end{enumerate}
\end{lemma}

\begin{proposition}[I$\!$I$\!$I Proposition 6.8 in \cite{Holm-thesis}]
\label{m=r-n+1:Basis}
The $S$-module $D^{(m)}(\cA)$ is free
with a basis
$
\{ P_\cH\delta_\cH^m\, |\, \cH\in \cA^{(n-1)}\}.
$
\end{proposition}

\begin{corollary}
\label{corollary:exp:m=r-n+1}
If $m=r-n+1$, then
$$
\exp D^{(m)}(\cA) = \{ m^{\binom{r}{m}}\}.
$$
\end{corollary}

\section{The case $m\geq r-n+1$}

In this section, we assume that $m\geq r-n+1$, and we
prove Holm's conjecture by giving a basis of $D^{(m)}(\cA)$.

Set
$$
\tilde{r}:= n+m-1,
$$
and add $\tilde{r}-r$ hyperplanes to $\cA=\{ H_1,\ldots, H_r\}$ so that
\begin{equation}
\widetilde{\cA}:=\cA\cup\{ H_{r+1}, \ldots, H_{\tilde{r}}\}
\end{equation}
is still generic.

For $\cH\in \widetilde{\cA}^{(n-1)}$, define a homogeneous polynomial
$P'_\cH\in S$ by
\begin{equation}
P'_\cH:= \prod_{H\notin\cH;\, H\in \cA} p_H.
\end{equation}

\begin{theorem}
\label{Holm'sConjecture}
The $S$-module $D^{(m)}(\cA)$ is free
with a basis
$
\{ P'_\cH\delta_\cH^m\, |\, \cH\in \widetilde{\cA}^{(n-1)}\}.
$
\end{theorem}

\begin{proof}
By \eqref{eqn:2.4}, $ P'_\cH\delta_\cH^m\in D^{(m)}(\cA)$ 
for each $\cH\in \widetilde{\cA}^{(n-1)}$.

By Lemma \ref{m=r-n+1:DualBasis} (1),
$\{ P'_\cH\delta_\cH^m\, |\, \cH\in \widetilde{\cA}^{(n-1)}\}$
is linearly independent over $S$.
Since
$$
\deg P'_\cH= \sharp\{ H\in\cA\,|\, H\notin \cH\},
$$
the number of $\cH\in \widetilde{\cA}^{(n-1)}$
with $\deg P'_\cH=j$
is
$$
\binom{r}{j}\binom{\tilde{r}-r}{n-1-(r-j)}
=
\binom{r}{j}\binom{m+n-r-1}{n-r+j-1}.
$$
Then
\begin{eqnarray*}
\sum_j j\binom{r}{j}\binom{m+n-r-1}{n-r+j-1}
&=&
r\sum_j \binom{r-1}{j-1}\binom{m+n-r-1}{n-r+j-1}\\
&=&
r\sum_j \binom{r-1}{j-1}\binom{m+n-r-1}{m-j}\\
&=&
r\binom{m+n-2}{m-1}
=r t_m.
\end{eqnarray*}
Hence we have the assertion by Theorem \ref{Theorem4.23}.
\end{proof}

\begin{corollary}
\label{cor:Exp:m>r-n}
$$
\exp D^{(m)}(\cA)
=
\{ j^{\binom{r}{j}\binom{m+n-r-1}{m-j}}\, |\, 
 r-n+1\leq j\leq \min\{ r, m\}
\}.
$$
\end{corollary}

\section{The case $m< r-n+1$}

Throughout this section, we assume that $m< r-n+1$.

Recall that $D^{(m)}(\cA)$ is generated by 
\begin{equation}
\label{eqn:generatorsHolm}
\{P_{\{\cH_1, \ldots, \cH_m\}}\delta_{\cH_1}\cdots\delta_{\cH_m}\mid 
\cH_1,\ldots, \cH_m\in\cA^{(n-1)}\}\cup\{\epsilon_m\}
\end{equation}
over $S$ (Theorem \ref{Theorem5.28}). 
In this section, 
we choose a minimal system of generators from 
\eqref{eqn:generatorsHolm} (Theorem \ref{theorem10.3}), which implies
that $D^{(m)}(\cA)$ is not free (Remark \ref{m<r-n+1:Nonfreeness}).

Note that
$$
\sharp\cA^{(n-1)}=\binom{r}{n-1}
>
\binom{n+m-1}{n-1}=s_m.
$$

\begin{lemma}
\label{m<r-n+1:Linear Dependence}
For any $\cH_1,\ldots,\cH_m\in \cA^{(n-1)}$, the following hold:
\begin{enumerate}
\item[{\rm (1)}]
${\displaystyle P_{\{\cH_1,\ldots,\cH_m\}}\in 
\bigcap_{\cap_{i=1}^m\cH_i\subset\cH\in\cA^{(n-1)}}SP_\cH}$.
\item[{\rm (2)}]
${\displaystyle \delta_{\cH_1}\cdots\delta_{\cH_m}\in 
\sum_{\cap_{i=1}^m\cH_i\subset\cH\in\cA^{(n-1)}}K\delta_\cH^m}$.
\end{enumerate}
\end{lemma}

\begin{proof}
{\rm (1)} If  $\cap_{i=1}^m\cH_i\subset \cH\in \cA^{(n-1)}$, 
then $P_{\cH}=\prod_{H\not\in\cH}p_H$ divides 
$\prod_{H\not\in\cap_{i=1}^m\cH_i}p_H=P_{\{\cH_1,\ldots,\cH_m\}}$. 
Hence the assertion is clear.

{\rm (2)} Let $\bar{r}:=n+m-1$. Take a subarrangement
$\cB\supset \cap_{i=1}^m\cH_i$ of $\cA$ with $\bar{r}$ hyperplanes. 
By Lemma \ref{m=r-n+1:DualBasis}, 
there exist $c_{\cH}\in K$ ($\cH\in \cB^{(n-1)}$) such that 
\begin{equation}
\label{eqn:DeltaHs}
\delta_{\cH_1}\cdots\delta_{\cH_m}=\sum_{\cH\in\cB^{(n-1)}}c_{\cH}\delta_{\cH}^m.
\end{equation}
It suffices to show that $c_{\cH}=0$ for all $\cH\not\supset \cap_{i=1}^{m}\cH_i$.
Fix $\cH\in \cB^{(n-1)}$ with
$\cH\not\supset \cap_{i=1}^{m}\cH_i$, 
and put $\overline{P}_\cH =\prod_{H\in \cB\setminus \cH} p_H$. 
Then $\deg \overline{P}_\cH=\bar{r}-(n-1)=m$.
Since there exists $H_0\in(\cap_{i=1}^{m}\cH_i)\setminus \cH$, 
we have $\delta_{\cH_i}\ast p_{H_0}=0$ for all $i=1,\ldots, m$,
and hence
$\delta_{\cH_1}\cdots\delta_{\cH_m}\ast\overline{P}_\cH =0$. 
Recall from \eqref{eqn:dual} that 
$$
\delta_{\cH'}^m\ast \overline{P}_{\cH}
=
\left\{
\begin{array}{ll}
m!\prod_{H\in\cB\setminus\cH}(\delta_{\cH'}\ast p_H)\ne 0 & \text{if $\cH'=\cH $}\\
0 & \text{otherwise}.
\end{array}
\right.
$$

Let the operator \eqref{eqn:DeltaHs} act on $\overline{P}_\cH$. Since
$$
0 = \delta_{\cH_1}\cdots\delta_{\cH_m}\ast\overline{P}_\cH
= \sum_{\cH\in\cB^{(n-1)}}c_{\cH}\delta_{\cH}^m\ast \overline{P}_\cH
= c_{\cH}\cdot m!\prod_{H\in\cB\setminus\cH}(\delta_{\cH}\ast p_H),
$$
we have $c_{\cH}=0$.
\end{proof}

\begin{proposition}
\label{m<r-n+1:SmallerGenerators}
If $m<r-n+1$, then
$$D^{(m)}(\cA)=\sum_{\cH\in\cA^{(n-1)}}SP_{\cH}\delta_{\cH}^m + S\epsilon_m$$
\end{proposition}

\begin{proof}
Let $\cH_{1},\ldots,\cH_{m}\in\cA^{(n-1)}$. 
By Lemma \ref{m<r-n+1:Linear Dependence}, 
\begin{eqnarray*}
P_{\{\cH_{1},\ldots,\cH_{m}\}}\delta_{\cH_1}\cdots\delta_{\cH_m} 
& \in & P_{\{\cH_{1},\ldots,\cH_{m}\}}\cdot \sum_{\cap_{i=1}^m\cH_i\subset\cH\in\cA^{(n-1)}}K\delta_\cH^m\\
& \subset & \sum_{\cap_{i=1}^m\cH_i\subset\cH\in\cA^{(n-1)}}S P_{\cH}\delta_\cH^m.
\end{eqnarray*}
Hence we obtain the assertion from \eqref{eqn:generatorsHolm}.
\end{proof}

The system of generators for $D^{(m)}(\cA)$ in 
Proposition \ref{m<r-n+1:SmallerGenerators} is still large. 
Next, {\bf we fix $m$ hyperplanes $H_1,\ldots,H_m$}, and
define an $S$-submodule 
$\Xi^{(m)}(\cA)$ of  $D^{(m)}(\cA)$ by
\begin{equation}
\label{eqn:Xi}
\Xi^{(m)}(\cA):=\{\theta\in D^{(m)}(\cA)\,|\, \theta\ast (p_{H_1}\cdots p_{H_m})=0\}.
\end{equation}
For $\cH\in\cA^{(n-1)}$ with $\cH\cap \{H_1, \ldots, H_m \}\ne\emptyset$,
we have $\delta_{\cH}\ast p_{H_i}=0$ for some $i$, and hence
$P_{\cH}\delta_{\cH}^m\in \Xi^{(m)}(\cA)$.
Furthermore we have the following.

\begin{theorem}
\label{theorem10.3}
If $m<r-n+1$, then
$$
D^{(m)}(\cA)=\Xi^{(m)}(\cA)\oplus S\epsilon_m
=\sum_{\substack{ \cH\in\cA^{(n-1)}\\ \cH\cap \{H_1, \ldots, H_m \}\ne\emptyset}}
SP_{\cH}\delta_{\cH}^m \oplus S\epsilon_m
$$
Moreover, the set 
$\{P_{\cH}\delta_{\cH}^m\,|\,  \cH\in\cA^{(n-1)}, \cH\cap \{H_1, \ldots, H_m \}\ne\emptyset\}$ 
is a minimal system of generators for $\Xi^{(m)}(\cA)$ over $S$.
\end{theorem}

\begin{proof}
Let $\theta\in D^{(m)}(\cA)$. 
Then $\theta-\frac{1}{m!}\,\frac{\theta\ast(p_{H_1}\cdots p_{H_m})}{p_{H_1}\cdots p_{H_m}}\epsilon_m\in\Xi^{(m)}(\cA)$, 
since $\theta\in D^{(m)}(\cA)\subset D^{(m)}(\langle p_{H_1}\cdots p_{H_m}\rangle)$
by \eqref{eqn:2.4}.
So we have $D^{(m)}(\cA)=\Xi^{(m)}(\cA)+ S\epsilon_m$.
Moreover, $\epsilon_m\ast (p_{H_1}\cdots p_{H_m})=m!p_{H_1}\cdots p_{H_m}\ne 0$ 
implies that $\Xi^{(m)}(\cA)\cap S\epsilon_m=0$.

Next, we show the second equality. By Proposition \ref {m<r-n+1:SmallerGenerators}, 
it suffices to show that 
$$P_{\cH_0}\delta^m_{\cH_0}\in \sum_{\substack{ \cH\in\cA^{(n-1)} \\\cH\cap \{H_1, \ldots, H_m \}\ne\emptyset}}SP_{\cH}\delta_{\cH}^m \oplus S\epsilon_m
$$ for every $\cH_0\in\cA^{(n-1)}$ with $\cH_0\cap\{H_1,\ldots,H_m\}=\emptyset$.
Put $\cB:=\cH_0\cup\{H_1, \ldots, H_m \}$. 
By Proposition \ref{m=r-n+1:Basis}, 
$$D^{(m)}(\cB)=\bigoplus_{\cH\in\cB^{(n-1)}}S\overline{P}_\cH\delta^m_{\cH},$$
where $\overline{P}_{\cH}=\prod_{H\in\cB\setminus\cH}p_H$.
Since $\epsilon_m\in D^{(m)}(\cB)$, 
there exist $c_{\cH}\in S$ ($\cH\in\cB^{(n-1)}$) such that 
\begin{equation}
\label{eqn:epsilon}
\epsilon_m =\sum_{\cH\in\cB^{(n-1)}}c_{\cH}\overline{P}_{\cH}\delta^m_{\cH}.
\end{equation}
(By looking at polynomial degrees, we see $c_{\cH}\in K$.)
Multiplying $q:=\prod_{H\in\cA\setminus\cB}p_H$ from the left, 
we have
\begin{equation}
\label{eqn:q epsilon} 
q\epsilon_{m}=\sum_{\cH\in\cB^{(n-1)}}c_{\cH}P_{\cH}\delta^m_{\cH}.
\end{equation}
Let the operator \eqref{eqn:epsilon} act on $\overline{P}_{\cH_0}$. 
Since
$$
0\ne m!\overline{P}_{\cH_0}=
m!c_{\cH_0}\overline{P}_{\cH_0}\cdot \prod_{H\in\cB\setminus\cH_0}\delta_{\cH_0}\ast p_H
=
m!c_{\cH_0}\overline{P}_{\cH_0}\cdot \prod_{i=1}^m(\delta_{\cH_0}\ast p_{H_i}),
$$
we have $c_{\cH_0}\ne 0$. 
Hence, we have
$$P_{\cH_0}\delta_{\cH_0}^m=c_{\cH_0}^{-1}\left(q\epsilon_m-\sum_{\substack{ \cH\in\cB^{(n-1)}\\ \cH\ne\cH_0}}c_{\cH}P_{\cH}\delta^m_{\cH}\right)\in \sum_{\substack{ \cH\in\cA^{(n-1)} \\\cH\cap \{H_1, \ldots, H_m \}\ne\emptyset}}SP_{\cH}\delta_{\cH}^m \oplus S\epsilon_m.$$

Finally, we show the minimality.
It suffices to show that 
the set 
$\{P_{\cH}\delta_{\cH}^m\,|\,  \cH\in\cA^{(n-1)}, \cH\cap \{H_1, \ldots, H_m \}\ne\emptyset\}$
 is linearly independent over $K$,
 since all $P_{\cH}\delta_{\cH}^m$ have the same polynomial degree.
Suppose that 
\begin{equation}
\label{eqn:LinearCombination}
\sum_{\substack{ \cH\in\cA^{(n-1)} \\\cH\cap \{H_1, \ldots, H_m \}\ne\emptyset}}
c_{\cH}P_{\cH}\delta_{\cH}^m=0\ (c_{\cH}\in K).
\end{equation}
Fix arbitrary hyperplanes $H_{i_1}, \ldots, H_{i_m}\in \cA$, and put
$q':=p_{H_{i_1}}\cdots p_{H_{i_m}}$ and $\cB':=\cA\setminus \{H_{i_1}, \ldots, H_{i_m}\}$.
Let the operator \eqref{eqn:LinearCombination} act on $q'$. 
Then we have
$$
\sum_{\substack{\cH\in\cB'^{(n-1)} \\ \cH\cap \{H_1, \ldots, H_m \}\ne\emptyset}}
c_{\cH}{P}_{\cH}\prod_{\nu=1}^{m}(\delta_\cH \ast p_{H_{i_{\nu}}})=0.
$$
By Lemma \ref{DualBasis}, the set $\{{P}_{\cH}\, |\, \cH\in\cB'^{(n-1)}\}$ 
is linearly independent over $K$.
Hence $c_{\cH}=0$ for $\cH\in\cB'^{(n-1)}$ with
$\cH\cap \{H_1, \ldots, H_m \}\ne\emptyset$.
For $\cH\in\cA^{(n-1)}$ with
$\cH\cap \{H_1, \ldots, H_m \}\ne\emptyset$, 
we may take $H_{i_1}, \ldots, H_{i_m}\in \cA$
so that $\cH\in\cB'^{(n-1)}$, since $r> m+n-1$.
Hence we have finished the proof. 
\end{proof}

\begin{corollary}[cf. Conjecture 6.8 in \cite{Snellman}]
\label{m<r-n+1:NumberOfGenerators}
The $S$-module
$\Xi^{(m)}(\cA)$ is minimally generated by
$\binom{r}{n-1}-\binom{r-m}{n-1}$ operators of polynomial degree $r-n+1$.
\end{corollary}

\begin{remark}
\label{m<r-n+1:Nonfreeness}
We can show
\begin{equation*}
\binom{r}{n-1}-\binom{r-m}{n-1}+1> \binom{n+m-1}{n-1},
\end{equation*}
supposing that $m<r-n+1$.
Then by Proposition \ref{Proposition4.18} and Corollary \ref{m<r-n+1:NumberOfGenerators}
we see that,
for $n\geq 3$ and $m<r-n+1$,
$D^{(m)}(\cA)$ is not free over $S$,
which was proved by Holm \cite[I$\!$I$\!$I Proposition 6.8]{Holm-thesis}.
\end{remark}


\section{Generalization of Yuzvinsky's paper \cite{Yuzvinsky}}

In this section, we assume
$m\leq r-n+1$, and we construct a minimal free resolution of $\Xi^{(m)}(\cA)$
when $m< r-n+1$ and $n\geq 3$.
We generalize the construction in \cite{Yuzvinsky}
step by step, and basically we succeed Yuzvinsky's notation.

Let $V:=K^n$.
Recall that, for $\cH\in \cA^{(n-1)}$,
 $\delta_\cH\in (V^\ast)^\ast=V$ is
the nonzero derivation with constant coefficients
such that $\delta_\cH\ast p_H=0$ for all $H\in \cH$.
Under the identification $(V^\ast)^\ast=V$,
$
K\delta_\cH$
corresponds to the linear subspace
$[\cH]:=\bigcap_{H\in\cH}H
=\bigcap_{H\in\cH}(p_H=0)$ of $V$.
Similarly,
$\cH\in \cA^{(n-j)}$
corresponds to the linear subspace
$[\cH]=\bigcap_{H\in\cH}H\in L_j$,
where $L_j$ is the set of elements of dimension $j$ of the intersection lattice
of $\cA$.

For $\cH\in \cA^{(n-j)}$ with $1\leq j\leq n$, set
$$
\Delta_\cH:=
\sum_{\cH'\in (\cA\setminus\cH)^{(j-1)}}K\delta_{\cH\cup\cH'}^m.
$$
Note that
$$
\Delta_\cH= K\delta^m_\cH\qquad
\text{for $\cH\in \cA^{(n-1)}$},
$$
and
$$
\Delta_\emptyset= \sum_{\cH\in \cA^{(n-1)}}K\delta_{\cH}^m.
$$
Each $\Delta_\cH$ is a subspace of $\Delta_\emptyset$.

\begin{example}
Let $m=1$.
Then
$$
\Delta_\cH=
\{
\delta\in (V^\ast)^\ast\,|\,
\delta\ast p_H=0
\quad
\text{for all $H\in \cH$}\}.
$$
Hence, under the identification $(V^\ast)^\ast=V$,
$\Delta_\cH$
corresponds to
$[\cH]=\bigcap_{H\in\cH}H
=\bigcap_{H\in\cH}(p_H=0)$.
\end{example}

\bigskip

\begin{lemma}
\label{lemma:DimOfDelta}
Let $1\leq j\leq n$, and
let $\cH\in \cA^{(n-j)}$.

Take $\cA':=\{ H_1, H_2,\ldots, H_{\bar{r}}\}\subseteq \cA$ 
with $\bar{r}=m+n-1$
so that $\cH\subseteq \cA'$.

Then $\{\delta_{\cH\cup\cH'}^m\, |\, \cH'\in (\cA'\setminus\cH)^{(j-1)}\}$
forms a basis of $\Delta_\cH$, and
$\dim \Delta_\cH = \binom{\bar{r}-(n-j)}{j-1}=\binom{m+j-1}{j-1}$.
\end{lemma}

\begin{proof}
By Lemma \ref{m=r-n+1:DualBasis}, 
\begin{equation}
\label{basis'}
\sum_{|\aalpha|=m}K \partial^\aalpha
=
\bigoplus_{\cH''\in (\cA')^{(n-1)}} K\delta_{\cH''}^m.
\end{equation}
Hence $\delta_{\cH\cup\cH'}^m$ $(\cH'\in (\cA'\setminus\cH)^{(j-1)})$
are linearly independent.

Let $\cH'''\in (\cA\setminus\cH)^{(j-1)} \setminus (\cA')^{(j-1)}$.
Then
\begin{equation}
\label{}
\delta_{\cH\cup\cH'''}^m
=
\sum_{\cH''\in (\cA')^{(n-1)}}
\frac{\delta_{\cH\cup\cH'''}^m \ast P'_{\cH''}}{\delta_{\cH''}^m \ast P'_{\cH''}}
\delta_{\cH''}^m,
\end{equation}
where
\begin{equation}
\label{P'}
P'_{\cH''}:=\prod_{H\in\cA'\setminus \cH''}p_H.
\end{equation}

For $\cH''\not\supseteq\cH$, there exists $H\in \cH\setminus \cH''$.
Then $p_H$ divides $P'_{\cH''}$, and hence
$\delta_{\cH\cup\cH'''}^m \ast P'_{\cH''}=0$.
Therefore
\begin{equation}
\label{}
\delta_{\cH\cup\cH'''}^m
=
\sum_{\cH'\in (\cA'\setminus\cH)^{(j-1)}}
\frac{\delta_{\cH\cup\cH'''}^m \ast P'_{\cH\cup\cH'}}{\delta_{\cH\cup\cH'}^m \ast P'_{\cH\cup\cH'}}
\delta_{\cH\cup\cH'}^m.
\end{equation}

Hence $\{\delta_{\cH\cup\cH'}^m\, |\, \cH'\in (\cA'\setminus\cH)^{(j-1)}\}$
forms a basis of $\Delta_\cH$, and
$\dim \Delta_\cH = \binom{\bar{r}-(n-j)}{j-1}=\binom{m+j-1}{j-1}$.
\end{proof}

\bigskip
Let $\cA=\{ H_1, H_2,\ldots, H_r\}$.
We write $H_i\prec H_j$ if $i<j$.

We define the complex $C_\ast(\cA)=C_\ast$ as follows.
For $j=1,2,\ldots, n$, set
$$
C_{n-j}:=
\bigoplus_{\cH\in\cA^{(n-j)}}
\Delta_\cH \e_{\wedge\cH},
$$
where $\e_{\wedge\cH}$ is just a symbol.
In particular,
$$
C_{n-1}:=
\bigoplus_{\cH\in\cA^{(n-1)}}
K\delta^m_\cH \e_{\wedge\cH},
$$
and
$$
C_{0}:=
\Delta_\emptyset \e_{\wedge\emptyset}.
$$
The differential
$\partial_j:C_j\to C_{j-1}$ is defined by
$$
C_{j}=
\bigoplus_{\cH\in\cA^{(j)}}
\Delta_\cH \e_{\wedge\cH}
\ni \xi \e_{\wedge\cH}
\mapsto
\sum_{H\in \cH} (-1)^{l_\cH(H)} \xi
\e_{\wedge(\cH\setminus\{ H\})}\in
C_{j-1},
$$
where
$$
l_\cH(H):=\sharp\{ H'\in \cH\, |\, H'\prec H\}.
$$
Set
$$
C_n:=\Ker\, \partial_{n-1}.
$$

\begin{lemma}[cf. Lemma 1.1 in \cite{Yuzvinsky}]
\label{Yuz:lemma1.1}
The sequence $C_\ast$ is exact.
\end{lemma}

\begin{proof}
As in \cite[Lemma 1.1]{Yuzvinsky},
we prove the assertion by induction.

Let $r=m+n-1$.
Then by Lemma \ref{m=r-n+1:DualBasis}
$$
\Delta_\cH=\bigoplus_{\cH'\in (\cA\setminus\cH)^{(j-1)}}
K\delta_{\cH\cup\cH'}^m
\qquad\text{for $\cH\in \cA^{(n-j)}$.}
$$
Hence
$$
C_{n-j}=\bigoplus_{\cH\in \cA^{(n-j)}}\bigoplus_{\cH'\in (\cA\setminus\cH)^{(j-1)}}
K\delta_{\cH\cup\cH'}^m \e_{\wedge\cH}
=
\bigoplus_{\cH\in \cA^{(n-1)}}
K\delta_\cH^m
\otimes
(\bigoplus_{\cH'\in \cH^{(n-j)}} K\e_{\wedge\cH'}).
$$
Thus, in this case, with $C_n=0$,
$$
C_\ast=
\bigoplus_{\cH\in \cA^{(n-1)}}
K\delta_\cH^m
\otimes \tilde{S}(\cH),
$$
where $\tilde{S}(\cH)$
is the augmented chain complex of the simplex with vertex set $\cH$.
Hence $C_\ast$ is exact.

For $n=2$,
the sequence
$$
\begin{CD}
0 @>>>
\Ker\,\partial_1
@>>>
C_1
@> \partial_1 >>
C_0 @>>> 0\\
@. @. @| @| @.\\
@. @. 
\bigoplus_{H\in\cA}K\delta_H^m
@>>>
\sum_{H\in\cA}K\delta_H^m
@. 
\end{CD}
$$
is clearly exact.

Suppose that $n>2$ and $r> m+n-1$.
Consider the arrangements
$\cA\setminus\{ H_r\}$ and
$\cA^{ H_r}$.
Since $r> m+n-1$, we have
$\Delta_\cH(\cA)=\Delta_{\cH}(\cA\setminus\{ H_r\})$
for $\cH\in (\cA\setminus\{ H_r\})^{(n-j)}$
by Lemma \ref{lemma:DimOfDelta}.
Hence
$$
0\to C_*(\cA\setminus\{ H_r\})
\to C_*(\cA)
\to C_*(\cA^{ H_r})(-1)
\to 0
$$
is exact.
We thus have the assertion by induction.
\end{proof}

Let $\cH\in\cA^{(n-j)}$ with $j=1,2,\ldots, n$,
and
let $C_\ast^{[\cH]}:=C_\ast(\cA^{[\cH]})$.
For $\cH'\in (\cA\setminus\cH)^{(j-t)}$,
we have
$$
\Delta_{\cH'}(\cA^{[\cH]})=
\sum_{\cH''\in (\cA\setminus \cH\cup\cH')^{(t-1)}}
K(\delta^{[\cH]}_{\cH'\cup\cH''})^m.
$$
Since we may identify $\delta^{[\cH]}_{\cH'\cup\cH''}$
with $\delta_{\cH\cup\cH'\cup\cH''}$,
we may identify $\Delta_{\cH'}(\cA^{[\cH]})$
with $\Delta_{\cH\cup\cH'}$. Hence
\begin{equation}
\label{eqn:IdentifyWithA''}
C_{j-t}^{[\cH]}=
\bigoplus_{\cH'\in (\cA\setminus\cH)^{(j-t)}}
\Delta_{\cH\cup\cH'} \e_{\wedge\cH'}\e_\cH
\end{equation}
for $t=1,2,\ldots, j$, where $\e_\cH$ is again a symbol.

We put
$$
E_{[\cH]}:=C_j^{[\cH]}
:=\Ker(\partial^{[\cH]}_{j-1}:
C_{j-1}^{[\cH]}\to
C_{j-2}^{[\cH]})
$$
for $\cH\in\cA^{(n-j)}$ with $j\geq 2$, and
$$
E_{[\cH]}:=
K\delta_\cH^m\e_\cH
$$
for $\cH\in\cA^{(n-1)}$.
Then we put
$$
E_j:=\bigoplus_{\cH\in\cA^{(n-j)}}E_{[\cH]}.
$$
for $j=1,2,\ldots, n$.

\begin{remark}[cf. Remark 1.2 in \cite{Yuzvinsky}]
\label{remark:Yuz1.2}
Let $1\leq j\leq n$ and $\cH\in\cA^{(n-j)}$. Then
$$
\dim E_{[\cH]}
=
\binom{r-m-n+j-1}{j-1}.
$$
\end{remark}

\begin{proof}
By Lemma \ref{Yuz:lemma1.1},
$$
\dim E_{[\cH]}
=
\dim C_j^{[\cH]}
=\sum_{l=1}^j(-1)^{l-1}
\dim C_{j-l}^{[\cH]}.
$$
Then by Lemma \ref{lemma:DimOfDelta}
$$
\dim C_{j-l}^{[\cH]}
=
\binom{r-n+j}{j-l}
\binom{m+l-1}{l-1}.
$$
Hence
\begin{eqnarray*}
\dim E_{[\cH]}
&=&
\sum_{l=1}^j(-1)^{l-1}
\binom{m+l-1}{l-1}
\binom{r-n+j}{j-l}\\
&=&
\sum_{l=1}^j(-1)^{l-1}
\binom{m+l-2}{l-1}
\binom{r-n+j-1}{j-l}
=\cdots \\
&=&
\sum_{l=1}^j(-1)^{l-1}
\binom{l-2}{l-1}
\binom{r-m-n+j-1}{j-l}
=
\binom{r-m-n+j-1}{j-1}.
\end{eqnarray*}
\end{proof}

Let 
$$
\Delta_{ij}:=\bigoplus_{\cH\in\cA^{(n-i)}}
\bigoplus_{\cH'\in(\cA\setminus\cH)^{(i+j-n)}}
\Delta_{\cH\cup\cH'}\e_{\wedge\cH'}\e_{\cH}
$$
for $1\leq i\leq n,\, 0\leq j\leq n-1$ with $i+j\geq n$,
and
$$
\Delta_{in}:= E_i=\bigoplus_{\cH\in\cA^{(n-i)}}E_{[\cH]}.
$$
Then
$$
\Delta_{ij}=\bigoplus_{\cH\in \cA^{(n-i)}} C_{i+j-n}^{[\cH]},
\quad
\text{and hence}
\quad
\Delta_{i \bullet}=\bigoplus_{\cH\in \cA^{(n-i)}} C_{\bullet}^{[\cH]}(-(n-i)).
$$
As differentials of $\Delta_{i \bullet}$,
we take $(-1)^i$ times the differentials of 
$\bigoplus_{\cH\in \cA^{(n-i)}} C_{\bullet}^{[\cH]}(-(n-i))$.
We define a linear map $\phi(j)_i:\Delta_{ij}\to \Delta_{i-1 j}$ 
for $0\leq j\leq n-1$ by
$$
\Delta_{ij}\ni
\xi \e_{\wedge\cH'}\e_{\cH}
\mapsto
\sum_{H\in \cH'} (-1)^{l_{\cH'}(H)}\xi \e_{\wedge(\cH'\setminus\{ H\})}
\e_{\cH\cup\{ H\}}
\in
\Delta_{i-1 j}
$$
for $\cH\in \cA^{(n-i)}, \cH'\in (\cA\setminus\cH)^{(i+j-n)}$,
and $\xi\in \Delta_{\cH\cup\cH'}$.
We define
$
\psi_i:E_i\to
E_{i-1}
$
as the restriction of $\phi(n-1)_i$.

Then we have
the double complex $\Delta_{\bullet \bullet}$:
$$
\tiny
\begin{CD}
@. 0 @. 0 @. 0 @. @. 0 @. 0 @. \\
@. @VVV @VVV @VVV @. @VVV @VVV @.\\
0 @>>> E_n @>>> \Delta_{n, n-1} @>>> \Delta_{n, n-2}
@>>> \cdots @>>> \Delta_{n, 1} @>>> \Delta_{n, 0} @>>> 0\\
@. @V{\psi_n}VV @VVV @VVV @. @VVV @VVV @.\\
0 @>>> E_{n-1} @>>> \Delta_{n-1, n-1} @>>> \Delta_{n-1, n-2}
@>>> \cdots @>>> \Delta_{n-1, 1} @>>> 0 @. \\
@. @V{\psi_{n-1}}VV @VVV @VVV @. @VVV @. @.\\
 @. \vdots @. \vdots @. \vdots
@.  @. 0 @.  @. \\
@. @V{\psi_3}VV @VVV @VVV @. @. @. @.\\
0 @>>> E_{2} @>>> \Delta_{2, n-1} @>>> \Delta_{2, n-2}
@>>> 0 @.  @. @. \\
@. @V{\psi_2}VV @VVV @VVV @. @. @. @.\\
0 @>>> E_{1} @>>> \Delta_{1, n-1} @>>> 0
@. @.  @. @. \\
@. @VVV @VVV @. @. @. @. @.\\
 @. 0 @. \,\, 0. @.
@. @.  @. @. \\
\end{CD}
$$

We add 
$$
\psi_1:
E_1=
\bigoplus_{\cH\in\cA^{(n-1)}}
K\delta_\cH^m\e_\cH
\ni \delta_\cH^m\e_\cH
\mapsto
\delta_\cH^m\in
E_0:=\Delta_\emptyset
=\sum_{\cH\in\cA^{(n-1)}}
K\delta_\cH^m.
$$

\begin{lemma}[cf. Lemma 1.3 in \cite{Yuzvinsky}]
\label{Yuz:lemma1.3}
The sequence 
$$
E_\ast:
0\to E_n\to E_{n-1} \to \cdots
\to E_1\to E_0\to 0
$$ 
is exact.
\end{lemma}

\begin{proof}
All rows of $\Delta_{\bullet \bullet}$ are exact by Lemma \ref{Yuz:lemma1.1}
and the argument in the paragraph just after the proof of 
Lemma \ref{Yuz:lemma1.1}.

For $1\leq j<n$, since we have
\begin{eqnarray*}
\Delta_{ij}
&=&
\bigoplus_{\cH\in\cA^{(n-i)}}
\bigoplus_{\cH'\in(\cA\setminus\cH)^{(i+j-n)}}
\Delta_{\cH\cup\cH'}\e_{\wedge\cH'}\e_{\cH}\\
&=&
\bigoplus_{\cH\in\cA^{(j)}}
\Delta_{\cH}
\otimes_K
(\bigoplus_{\cH'\in \cH^{(i+j-n)}}
K \e_{\wedge\cH'}\e_{\cH\setminus\cH'}),
\end{eqnarray*}
the $j$-th column $\Delta_{\bullet j}$ is the same as 
$\bigoplus_{\cH\in\cA^{(j)}} \Delta_{\cH}\otimes_K \tilde{S}(\cH)$,
where $\tilde{S}(\cH)$ is the augmented chain complex of the simplex with vertex set $\cH$:
$$
0\to K\e_{\wedge\cH}\to
\bigoplus_{B\in{\cH}^{(j-1)}}
K\e_{\wedge B}\to
\bigoplus_{B\in{\cH}^{(j-2)}}
K\e_{\wedge B}
\to\cdots\to
\bigoplus_{H\in{\cH}}
K\e_{H}\to K\e_\emptyset\to 0.
$$
Thus the $j$-th columns ($1\leq j\leq n-1$)
are exact.
The $0$-th column has the unique nonzero term $\Delta_\emptyset \e_\emptyset(=E_0)$
at $i=n$.
Hence by the spectral sequence argument we see that $E_\ast$ is exact.
\end{proof}

\bigskip
Let $\sigma\subseteq \{ 1,2,\ldots, r\}$ and $\sigma\neq\emptyset$.
Put
$$
\cL_j[\sigma]:=\{\, \cH\in \cA^{(n-j)}\,\, | \,\, 
\cH\cap\{ H_i\,|\, i\in \sigma\}\neq\emptyset\,\}.
$$
For $1\leq j\leq n$, 
$$
E_j[\sigma]:=
\bigoplus_{\cH\in \cL_j[\sigma]} E_{[\cH]}.
$$
Then
$$
E_n[\sigma]=0,
\qquad
E_1[\sigma]=\bigoplus_{\cH\in \cL_1[\sigma]}K\delta_{\cH}^m\e_\cH.
$$
We put
$$
E_0[\sigma]:=\sum_{\cH\in \cL_1[\sigma]}K\delta_{\cH}^m.
$$
We also put
$$
E_j[\emptyset]:=0.
$$
for all $j$.
Then
$\{ (E_\ast[\sigma], \psi_\ast[\sigma])\}$ is a subcomplex of
$\{ (E_\ast, \psi_\ast)\}$.

\begin{lemma}[cf. Lemma 1.4 in \cite{Yuzvinsky}]
\label{Yuz:lemma1.4}
For every $\sigma$ with $|\sigma|\leq n+m-1$,
$E_\ast[\sigma]$ is exact.
\end{lemma}

\begin{proof}
We prove the assertion by induction on $|\sigma|$.
If $|\sigma|=0$, then the assertion is trivial.

When $n=2$,
we have
$$
0\to
E_1[\sigma]=\bigoplus_{H\in \cL_1[\sigma]}K\delta^m_H\e_H
=\bigoplus_{H\in \sigma}K\delta^m_H\e_H
\to
E_0[\sigma]=\sum_{H\in \cL_1[\sigma]}K\delta^m_H
=\sum_{H\in \sigma}K\delta^m_H
\to 0.
$$
This is an isomorphism, since $|\sigma|\leq 2+m-1$
(see Lemma \ref{m=r-n+1:DualBasis}).

Now assume that $|\sigma|\geq 1$ and $n\geq 3$.
Fix $j\in \sigma$ and put $\tau:=\sigma\setminus\{ j\}$.
Then $E_\ast[\tau]$ and $E_\ast[\{ j\}]=E_\ast(\cA^{H_j})$ 
(by \eqref{eqn:IdentifyWithA''})
are subcomplexes of $E_\ast[\sigma]$,
which are exact by the induction hypothesis and Lemma \ref{Yuz:lemma1.3}.
Moreover there exists an exact sequence of complexes:
$$
0\to
E_\ast[\tau]\cap E_\ast[\{ j\}]
\to
E_\ast[\tau]\oplus E_\ast[\{ j\}]
\to
E_\ast[\sigma]
\to
0.
$$
Since
$E_\ast[\tau]\cap E_\ast[\{ j\}]=E_\ast[\tau](\cA^{H_j})$
and $|\tau|\leq (n-1)+m-1$, we are done.
\end{proof}

\bigskip

Put
$$
\sigma_0:=\{ 1, 2,\ldots, m\},
$$
and
$$
\bar{E}_\ast:=E_\ast[\sigma_0].
$$
We use notation
$$
\psi_j: \bar{E}_j\to \bar{E}_{j-1}\qquad
(j=1,2,\ldots, n-1).
$$

Put
$$
F_j:=S\otimes\bar{E}_j \qquad (j=0,1,\ldots, n-1).
$$
Note that
$F_i$ is a submodule of 
$$
S\Delta_{i, n-1}[\sigma_0]
=
\bigoplus_{\cH\in\cL_i[\sigma_0]}
\bigoplus_{\cH'\in (\cA\setminus\cH)^{(i-1)}}
S \delta_{\cH\cup\cH'}^m \e_{\wedge\cH'}\e_\cH.
$$
For $i\geq 2$, the morphism $d_i: F_i\to F_{i-1}$
is defined by
$$
\e_{\wedge\cH'}\e_\cH
\mapsto
\sum_{H'\in\cH'}
(-1)^{l_{\cH'}(H')}
p_{H'}
\e_{\wedge(\cH'\setminus\{ H'\})}
\e_{\cH\cup \{ H'\}}.
$$
Note that
$$
F_0=S\otimes_K\sum_{\cH\in\cL_1[\sigma_0]}K\delta_\cH^m,
$$
and
$$
F_1=\bigoplus_{\cH\in\cL_1[\sigma_0]}S\delta_\cH^m
\e_\cH.
$$
We define a morphism
$d_1:F_1\to F_0$
by
$$
\delta_\cH^m
\e_\cH\mapsto
P_\cH\delta_\cH^m.
$$

\begin{lemma}
The sequence 
$$
0\to
F_{n-1}\overset{d_{n-1}}{\to}
F_{n-2}\overset{d_{n-2}}{\to}
\cdots
\overset{d_{2}}{\to}
F_{1}\overset{d_{1}}{\to}
F_0\to 0
$$ 
is a complex.
\end{lemma}

\begin{proof}
By the definition of $d_i$, clearly $d_{i}\circ d_{i+1}=0$ for $i\geq 2$.
We prove $d_1\circ d_2=0$.

Let $X=\sum_{\cH\in \cL_2[\sigma_0]}\sum_{H\notin\cH}
f_{\cH, H}\delta_{\cH\cup\{ H\}}^m \e_{\wedge H}\e_\cH\in F_2$.
Then
\begin{equation*}
\label{XinF_2}
\sum_{H\notin\cH}
f_{\cH, H}\delta_{\cH\cup\{ H\}}^m=0
\quad
\text{for all $\cH\in \cL_2[\sigma_0]$.}
\end{equation*}
We have
\begin{eqnarray*}
d_1\circ d_2(X)
&=&
d_1(\sum_{\cH\in \cL_2[\sigma_0]}\sum_{H\notin\cH}
f_{\cH, H}p_H\delta_{\cH\cup\{ H\}}^m \e_{\cH\cup\{ H\}})\\
&=&
\sum_{\cH\in \cL_2[\sigma_0]}\sum_{H\notin\cH}
f_{\cH, H}p_HP_{\cH\cup\{ H\}}\delta_{\cH\cup\{ H\}}^m\\
&=&
\sum_{\cH\in \cL_2[\sigma_0]}\sum_{H\notin\cH}
f_{\cH, H}P_{\cH}\delta_{\cH\cup\{ H\}}^m\qquad 
(\text{Here $P_{\cH}:=\prod_{H\notin \cH}p_H$.})\\
&=&
\sum_{\cH\in \cL_2[\sigma_0]}P_{\cH}\sum_{H\notin\cH}
f_{\cH, H}\delta_{\cH\cup\{ H\}}^m=0.
\end{eqnarray*}
\end{proof}

The following is Theorem \ref{theorem10.3}.

\begin{lemma}[cf. Lemma 2.1 in \cite{Yuzvinsky}]
Assume that $m< r-n+1$.
Then the image of $d_1$ coincides with $\Xi^{(m)}(\cA)$.
\end{lemma}

By Remark \ref{remark:Yuz1.2},
we have the following.

\begin{remark}[cf. Remark 2.2 in \cite{Yuzvinsky}]
\label{remark2.2}
$$
\rank_S(F_j)=
\binom{r-m-n+j-1}{j-1}
\left(
\binom{r}{n-j}
-
\binom{r-m}{n-j}
\right)=:w_j^{(m)}.
$$
\end{remark}

Under the above preparations, we can prove the following theorem.
Since the proof is almost the same as that of \cite[Theorem 2.3]{Yuzvinsky},
we omit it.

\begin{theorem}[cf. Theorem 2.3 in \cite{Yuzvinsky}]
\label{thm:MinFreeRes}
Assume that $n\geq 3$ and $m< r-n+1$.
Then the complex 
$$
F_*:
0\to
F_{n-1}\overset{d_{n-1}}{\to}
F_{n-2}\overset{d_{n-2}}{\to}
\cdots
\overset{d_{2}}{\to}
F_{1}\overset{d_{1}}{\to}
\Xi^{(m)}(\cA)\to 0
$$ 
is a minimal free resolution of
$\Xi^{(m)}(\cA)$.
In particular,
the projective dimensions of $S$-modules
$\Xi^{(m)}(\cA)$ and $D^{(m)}(\cA)$ are equal to $n-2$.
\end{theorem}

By Theorem \ref{theorem10.3}, Remark \ref{remark2.2},
and the construction of the complex $F_*$ in Theorem \ref{thm:MinFreeRes},
we have the following corollary:

\begin{corollary}[cf. Corollary 4.4.3 in \cite{Rose-Terao}]
\label{Rose-Terao:Cor4.4.3}
Assume that $n\geq 3$ and $m< r-n+1$.
Then there exist exact sequences
\begin{eqnarray*}
0 &\to& S(m+1-r)^{w_{n-1}^{(m)}} \to \cdots \to S(m+n-j-r)^{w_j^{(m)}}\to\cdots\\
&\to& S(m+n-2-r)^{w_2^{(m)}} \to S(m+n-1-r)^{w_1^{(m)}}\to \Xi^{(m)}(\cA)\to 0,\\
0 &\to& S(m+1-r)^{w_{n-1}^{(m)}} \to \cdots \to S(m+n-j-r)^{w_j^{(m)}}\to\cdots\\
&\to& S(m+n-2-r)^{w_2^{(m)}} \to S(m+n-1-r)^{w_1^{(m)}}\bigoplus S \to D^{(m)}(\cA)\to 0,
\end{eqnarray*}
where $w_j^{(m)}$ were defined in Remark \ref{remark2.2}, and all maps are homogeneous of degree $0$.

In particular,
the Castelnuovo-Mumford regularities of $\Xi^{(m)}(\cA)$ and $D^{(m)}(\cA)$ are equal to $r-m-n+1$.
\end{corollary}

\begin{remark}
If we use the polynomial degrees in $\Xi^{(m)}(\cA)$ and $D^{(m)}(\cA)$
as the degrees of graded $S$-modules, then
the degrees are shifted by $m$.
Then the Castelnuovo-Mumford regularities of $\Xi^{(m)}(\cA)$ and $D^{(m)}(\cA)$ 
are equal to $r-n+1$ as stated for $D^{(1)}(\cA)$ in 
\cite[Section 5.2]{Derken-Sidman}, and
the Poicar\'e-Betti series of $\Xi^{(m)}(\cA)$ and $D^{(m)}(\cA)$
coincide with the ones 
conjectured by Snellman \cite[Conjecture 6.8]{Snellman}.
\end{remark}

\section{Minimal free resolution of $J_m(\cA)$}
\label{Jacobian}

In this section, we generalize the minimal free resolution of $S/J$
given in \cite{Rose-Terao}, where
$J$ is the Jacobian
ideal of $Q$.
We retain the assumptions $n\geq 3$ and $n+m-1< r$.

Let $J_m(\cA)$ denote the $S$-submodule of 
$S^{\binom{n+m-1}{m-1}}=\bigoplus_{|\bbeta|\leq m-1}S\e_\bbeta$
generated by all 
\begin{equation}
\label{ActionBullet}
\frac{1}{\aalpha !}\partial^\aalpha\bullet Q:=
(\frac{1}{(\aalpha-\bbeta)!} \partial^{\aalpha-\bbeta}* Q\, :\, |\bbeta|\leq m-1)
=
\sum_{|\bbeta|\leq m-1}\frac{1}{(\aalpha-\bbeta)!} \partial^{\aalpha-\bbeta}* Q
\e_\bbeta
\end{equation}
with $1\leq |\aalpha|\leq m$.
Here we agree $\partial^{\aalpha-\bbeta}=0$
for $\bbeta\not\leq\aalpha$.

\begin{example}
Let $m=1$.
Then
$J_1(\cA)$ is the $S$-submodule of $S$
generated by
$\partial_j *Q$ ($j=1,\ldots,n$),
i.e.,
$
J_1(\cA)
$
is nothing but the Jacobian ideal $J$ of $Q$.
\end{example}

\begin{lemma}
\label{lemma:partial=ad}
For all $\aalpha, \bbeta\in \N^n$,
$$
\frac{1}{(\aalpha-\bbeta)!}
\partial^{\aalpha-\bbeta}
=(-1)^{|\bbeta|}\frac{({\rm ad} x)^\bbeta}{\aalpha !}(\partial^\aalpha).
$$
Here we denote by ${\rm ad} x_i$ the endomorphism of $D(S)$:
$D(S)\ni P\mapsto {\rm ad} x_i(P)=[ x_i, P]\in D(S)$.
For $\bbeta=(\beta_1,\ldots, \beta_n)\in \N^n$,
we set $({\rm ad} x)^\bbeta=({\rm ad} x_1)^{\beta_1}\circ
\cdots\circ({\rm ad} x_n)^{\beta_n}$.
\end{lemma}

\begin{proof}
We prove the assertion by induction on $|\bbeta|$.
For all $\aalpha, \bbeta\in \N^n$, 
\begin{eqnarray*}
{\rm ad} x_i(
(-1)^{|\bbeta|}\frac{({\rm ad} x)^\bbeta}{\aalpha !}(\partial^\aalpha))
&=&
\frac{1}{(\aalpha-\bbeta)!}
{\rm ad} x_i(
\partial^{\aalpha-\bbeta})\\
&=&
-\frac{1}{(\aalpha-\bbeta)!}
(\alpha_i-\beta_i)
\partial^{\aalpha-\bbeta-\1_i}\\
&=&
-\frac{1}{(\aalpha-\bbeta-\1_i)!}
\partial^{\aalpha-\bbeta-\1_i}.
\end{eqnarray*}
\end{proof}

By Lemma \ref{lemma:partial=ad},
$$
\frac{1}{\aalpha !}\partial^\aalpha\bullet Q
=
((-1)^{|\bbeta|}({\rm ad} x)^\bbeta (\frac{1}{\aalpha !}\partial^\aalpha)* Q
\, :\, |\bbeta|\leq m-1).
$$
We define an $S$-module morphism 
$$
\delta_0: F_0^{[1,m]}:=D^{[1,m]}(S):=
\bigoplus_{k=1}^m D^{(k)}(S)\to S^{\binom{n+m-1}{m-1}}=\bigoplus_{|\bbeta|\leq m-1}S\e_\bbeta
$$
by
\begin{equation}
\label{def:delta0}
\delta_0(\theta)
:=
\theta\bullet Q
:=
((-1)^{|\bbeta|}({\rm ad} x)^\bbeta (\theta)* Q
\, :\, |\bbeta|\leq m-1).
\end{equation}

By definition,
\begin{equation}
\label{eqn:Im(delta_0)}
\Im\, \delta_0 = J_m(\cA).
\end{equation}

\begin{lemma}
\label{lemma:ThetaXbbeta}
Let $\theta\in D(S)$. Then
$$
\theta x^\bbeta
=
\sum_{\ggamma\leq \bbeta}(-1)^{|\ggamma|}
\binom{\bbeta}{\ggamma} x^{\bbeta-\ggamma}({\rm ad} x)^\ggamma (\theta),
$$
where
$\binom{\bbeta}{\ggamma}=\prod_{i=1}^n \binom{\beta_i}{\gamma_i}$.
\end{lemma}

\begin{proof}
We prove the assertion by induction on $|\bbeta|$.
We have
\begin{eqnarray*}
\theta x_i x^\bbeta
&=&
-({\rm ad} x_i(\theta))x^\bbeta
+x_i\theta x^\bbeta\\
&=&
-\sum_{\ggamma\leq\bbeta}
(-1)^{|\ggamma|}
\binom{\bbeta}{\ggamma} x^{\bbeta-\ggamma}({\rm ad} x)^\ggamma ({\rm ad} x_i(\theta))
+
x_i\sum_{\ggamma\leq \bbeta}(-1)^{|\ggamma|}
\binom{\bbeta}{\ggamma} x^{\bbeta-\ggamma}({\rm ad} x)^\ggamma (\theta)\\
&=&
\sum_{\ggamma\leq\bbeta}
(-1)^{|\ggamma+\1_i|}
\binom{\bbeta}{\ggamma} x^{\bbeta+\1_i-\ggamma-\1_i}({\rm ad} x)^{\ggamma+\1_i}(\theta)
+
\sum_{\ggamma\leq \bbeta}(-1)^{|\ggamma|}
\binom{\bbeta}{\ggamma} x^{\bbeta+\1_i-\ggamma}({\rm ad} x)^\ggamma (\theta)\\
&=&
\sum_{\ggamma-\1_i\leq\bbeta}
(-1)^{|\ggamma|}
\binom{\bbeta}{\ggamma-\1_i} x^{\bbeta+\1_i-\ggamma}({\rm ad} x)^{\ggamma}(\theta)
+
\sum_{\ggamma\leq \bbeta}(-1)^{|\ggamma|}
\binom{\bbeta}{\ggamma} x^{\bbeta+\1_i-\ggamma}({\rm ad} x)^\ggamma (\theta)\\
&=&
\sum_{\ggamma\leq \bbeta+\1_i}(-1)^{|\ggamma|}
\binom{\bbeta+\1_i}{\ggamma} x^{\bbeta+\1_i-\ggamma}({\rm ad} x)^\ggamma (\theta).
\end{eqnarray*}
\end{proof}

Let $\bar{\delta}_0$ denote the composite of $\delta_0$ with
the canonical projections of $S\e_\bbeta$ onto
$(S/SQ)\e_\bbeta$ for $\bbeta\neq\0$:
\begin{equation}
\bar{\delta}_0:
D^{[1,m]}(S)\overset{\delta_0}{\to} \bigoplus_{|\bbeta|\leq m-1}S\e_\bbeta
\to S\e_\0\bigoplus\bigoplus_{0\neq |\bbeta|\leq m-1}(S/SQ)\e_\bbeta.
\end{equation}
Here note that $\bar{\delta}_0$ is a graded $S$-module homomorphism homogeneous
of degree $0$ if we put $\deg(\e_\bbeta)=-r-|\bbeta|$.

In the following two lemmas, we describe the cokernel and the kernel
of $\bar{\delta}_0$.

\begin{lemma}
\label{Im(bar-delta_0)}
$$
\Coker\, \bar{\delta}_0= S^{\binom{n+m-1}{m-1}}/
(J_m(\cA)+Q S^{\binom{n+m-1}{m-1}}).
$$
\end{lemma}

\begin{proof}
By \eqref{eqn:Im(delta_0)}, we only need to show
$Q\e_\0\in \Im\, \bar{\delta}_0$.
We have $\epsilon_1 * Q= rQ$.
Since $\epsilon_1\in D(\cA)$, we see
$\delta_0(\epsilon_1)\in \bigoplus_{|\bbeta|\leq m-1}SQ\e_\bbeta$
by the definition of $\delta_0$ \eqref{def:delta0}.
Hence
$$
Q\e_\0=\bar{\delta}_0(\frac{1}{r}\epsilon_1)
\in \Im\, \bar{\delta}_0.
$$
\end{proof}

\begin{lemma}
\label{lem:KerJ}
$$
\Ker\, \bar{\delta}_0
=\bigoplus_{k=1}^m D^{(k)}(\cA)'=: D^{[1,m]}(\cA)',
$$
where
$D^{(k)}(\cA)'
:=\{ \theta\in D^{(k)}(\cA)\, :\, \theta\ast Q=0\}$.
\end{lemma}

\begin{proof}
If $\theta\in D^{[1,m]}(\cA)'$, then
$({\rm ad} x)^\bbeta(\theta)\in D(\cA)$
for all $\bbeta$, and $\theta * Q=0$.
Hence $\theta\in \Ker\, \bar{\delta}_0$
by the definitions of $D(\cA)$ and $\bar{\delta}_0$.

Next we suppose that $\theta\in \Ker\, \bar{\delta}_0$.
Then by Lemma \ref{lemma:ThetaXbbeta}
\begin{equation}
\label{claim:J_m}
\text{
$\theta* x^\bbeta Q\in \langle Q\rangle = QS$ for all $\bbeta$ with $|\bbeta|\leq m-1$}.
\end{equation}
By \cite[Proposition 2.3]{Holm04},
we conclude that $\theta\in D^{[1,m]}(\cA)'$.
\end{proof}

\begin{lemma}
\label{Rose-Terao:Lemma4.5.2(2)}
Let $k\leq r$.
As $S$-modules,
$$
\Xi^{(k)}(\cA)\simeq D^{(k)}(\cA)'.
$$
\end{lemma}

\begin{proof}
It is easy to see that
$$
\gamma_k:
\Xi^{(k)}(\cA)
\ni \theta \mapsto \theta-\frac{\theta\ast Q}{Q}
\frac{\epsilon_k}{r(r-1)\cdots (r-k+1)}\in
 D^{(k)}(\cA)'
$$
and 
$$
D^{(k)}(\cA)'\ni \theta \mapsto \theta-
\frac{\theta\ast(p_1\cdots p_k)}{p_1\cdots p_k}
\frac{\epsilon_k}{k!}\in \Xi^{(k)}(\cA)
$$
are inverse to each other.
\end{proof}

For $1\leq k\leq m$,
let $F_*^{(k)}$ denote the minimal free resolution of
$\Xi^{(k)}$ in Theorem \ref{thm:MinFreeRes}.

We consider the following complex:
\begin{equation}
\label{eqn:MinResOfJm}
0\to 
\tilde{F}_{n-1}
\overset{\tilde{{\delta}}_{n-1}}{\to}
\cdots
\overset{\tilde{{\delta}}_{2}}{\to}
\tilde{F}_{1}
\overset{\tilde{{\delta}}_{1}}{\to}
\tilde{F}_{0}
\overset{\tilde{{\delta}}_{0}}{\to}
\tilde{F}_{-1}
\to
\Coker({\tilde{\delta}}_{0})
\to
0,
\end{equation}
where
\begin{eqnarray*}
\tilde{F}_{-1}
&=&
\bigoplus_{|\bbeta|\leq m-1}S\e_\bbeta,\\
\tilde{F}_{0}
&=&
D^{[1,m]}(S)
\bigoplus
\bigoplus_{0\neq |\bbeta|\leq m-1}S\e_\bbeta,\\
\tilde{F}_{j}
&=&
\bigoplus_{k=1}^m F_j^{(k)}
\qquad
(j=1,\ldots, n-1),
\end{eqnarray*}
and
\begin{eqnarray*}
\tilde{\delta}_{0}
(\theta, \sum_{\bbeta\neq \0} f_\bbeta\e_\bbeta)
&=&
\delta_0(\theta)+\sum_{\bbeta\neq\0}  f_\bbeta Q\e_\bbeta= \theta * Q \e_\0+
\sum_{\bbeta\neq\0} 
( (-1)^{|\bbeta|}({\rm ad} x)^\bbeta(\theta)* Q +f_\bbeta Q)
\e_\bbeta,\\
\tilde{\delta}_{1}
(\delta_\cH^k\e_\cH^{(k)})
&=&
(\gamma_k(P_\cH\delta_\cH^k),
-\frac{1}{Q}\sum_{\bbeta\neq\0}
(-1)^{|\bbeta|}({\rm ad} x)^\bbeta
(\gamma_k(P_\cH\delta_\cH^k))* Q \e_\bbeta),\\
\tilde{\delta}_{j}
&=&
\bigoplus_{k=1}^m {d}_{j}^{(k)}
\qquad (j\geq 2).
\end{eqnarray*}
Recall that
$D^{(k)}(S)=F_0^{(k)}$, and 
$d_1(\delta_\cH^k\e_\cH^{(k)})=P_\cH\delta_\cH^k$ for $1\leq k\leq m$.

\begin{theorem}[cf. Theorem 4.5.3 in \cite{Rose-Terao}]
\label{Rose-Terao:Theorem4.5.3}
The complex \eqref{eqn:MinResOfJm}
is a minimal free resolution of
$\Coker (\bar{\delta}_0)= S^{\binom{n+m-1}{m-1}}/(J_m(\cA)+Q S^{\binom{n+m-1}{m-1}})$.
\end{theorem}

\begin{proof}
The complex \eqref{eqn:MinResOfJm}
is exact by
Theorem \ref{theorem10.3},
Theorem \ref{thm:MinFreeRes},
Lemma \ref{Im(bar-delta_0)},
Lemma \ref{lem:KerJ},
and Lemma \ref{Rose-Terao:Lemma4.5.2(2)}.
The operator
$P_\cH \delta_\cH^k$
is of order $k$ and homogeneous of polynomial degree 
$\deg(P_\cH)= r-(n-1)$.
Then each term of $\gamma_k(P_\cH \delta_\cH^k)$
is of order $k$ and of polynomial degree greater than or equal to $k$.
Hence each term of
the operator
$({\rm ad} x)^\bbeta(\gamma_k(P_\cH \delta_\cH^k))$
is of order $k-|\bbeta|$ and of polynomial degree 
greater than or equal to $k$.
Therefore each term of the polynomial
$$
\frac{1}{Q}
(-1)^{|\bbeta|}({\rm ad} x)^\bbeta
(\gamma_k(P_\cH\delta_\cH^k))* Q
$$
is of degree greater than or equal to
$$
r-(k-|\bbeta|)+k-r
= |\bbeta|>0.
$$
Thus the free resolution \eqref{eqn:MinResOfJm} of $\Coker({\tilde{\delta}}_{0})$
is minimal.
Clearly by \eqref{Im(bar-delta_0)}
$$
\Coker({\tilde{\delta}}_{0})
= S^{\binom{n+m-1}{m-1}}/(J_m(\cA)+Q S^{\binom{n+m-1}{m-1}})=
\Coker (\bar{\delta}_0).
$$
\end{proof}

The following corollary is clear from Theorem \ref{Rose-Terao:Theorem4.5.3}
and the Auslander-Buchsbaum formula.

\begin{corollary}[cf. Corollary 4.5.5 \cite{Rose-Terao}]
The projective dimension of the $S$-module 
$S^{\binom{n+m-1}{m-1}}/(J_m(\cA)+Q S^{\binom{n+m-1}{m-1}})$
is $n$, and the depth is $0$.
\end{corollary}

In the complex \eqref{eqn:MinResOfJm}, the degrees of elements of bases are
as follows:
$$
\begin{array}{ll}
\deg(\e_\bbeta)= -r-|\bbeta| & \text{in $\tilde{F}_{-1}$,}\\
\deg(\partial^\aalpha)= -|\aalpha| & \text{in $\tilde{F}_{0}$,}\\
\deg(\e_\bbeta)= -|\bbeta| & \text{in $\tilde{F}_{0}$,}\\
\deg(\delta_\cH^k\e_\cH)= -k+r-(n-1)=r-n-k+1 & \text{in $\tilde{F}_{1}$.}\\
\end{array}
$$
Hence we have the following corollary:

\begin{corollary}[cf. Corollary 4.5.4 in \cite{Rose-Terao}]
\label{Rose-Terao:Cor4.5.4}
Assume that $n\geq 3$ and $m< r-n+1$.
Then there exists an exact sequence
\begin{eqnarray*}
0 &\to& \bigoplus_{k=1}^m S(k+1-r)^{w_{n-1}^{(k)}} \to \cdots \to 
\bigoplus_{k=1}^m S(k+n-j-r)^{w_j^{(k)}}\to\cdots\\
&\to& \bigoplus_{k=1}^m S(k+n-1-r)^{w_1^{(k)}}
\to \bigoplus_{k=1}^m S(k)^{s_k}\bigoplus \bigoplus_{k=1}^{m-1} S(k)^{s_k}
\to\\
&&
\bigoplus_{k=0}^{m-1} S(r+k)^{s_k}
 \to \Coker (\bar{\delta}_0)\to 0,
\end{eqnarray*}
where $w_j^{(k)}$ were defined in Remark \ref{remark2.2}, 
$s_k=\binom{n+k-1}{k}$,
and all maps are homogeneous of degree $0$.

In particular,
the Castelnuovo-Mumford regularity of 
$\Coker (\bar{\delta}_0)$ is equal to $r-n-2$.
\end{corollary}

\begin{remark}
In Corollary \ref{Rose-Terao:Cor4.5.4},
to make the degrees of all the minimal generators of
$\Coker (\bar{\delta}_0)$
nonnegative, we can shift the degrees by $r+(m-1)$ as
in \cite[Corollary 4.5.5]{Rose-Terao}.
Then the Castelnuovo-Mumford regularity of 
$\Coker (\bar{\delta}_0)$ is equal to $2r+m-n-3$.
\end{remark}

\section{Jet modules}

In this section,
we prove that
$\Coker (\bar{\delta}_0)=S^{\binom{n+m-1}{m-1}}/(J_m(\cA)+Q S^{\binom{n+m-1}{m-1}})$ in Section \ref{Jacobian}
is the transpose of the $m$-jet module
$\Omega^{[1,m]}(S/SQ)$.
For the basics of jet modules, see \cite{ega4}, \cite{Heyneman-Sweedler}, and \cite{Sw}.

Let $I:=\langle f_1,\ldots, f_k\rangle$ be an ideal of $S$.
Let $R:=S/I$.
Define jet modules
\begin{equation}
\begin{split}
\Omega^{[1,m]}(S)&:=J_S/J_S^{m+1},\qquad \Omega^{\leq m}(S):=S\otimes_K S/J_S^{m+1},\\
\Omega^{[1,m]}(R)&:=J_R/J_R^{m+1},\qquad \Omega^{\leq m}(R):=R\otimes_K R/J_R^{m+1},
\end{split}
\end{equation}
where
\begin{eqnarray*}
J_S&:=& \langle 1\otimes a-a\otimes 1\, |\, a\in S\rangle
\subseteq S\otimes_K S,\\
J_R&:=& \langle 1\otimes a-a\otimes 1\, |\, a\in R\rangle
\subseteq R\otimes_K R.
\end{eqnarray*}

Then $\Omega^{\leq m}(R)$ is the representative object
of the functor
$M\to D_R^m( R, M)$, i.e.,
there exists a natural isomorphism of $R$-modules:
$$
D_R^m( R, M)\simeq \Hom_R( \Omega^{\leq m}(R), M),
$$
where $M$ is an $R$-module, and $D_R^m( R, M)$
is the module of differential operators of order $\leq m$
from $R$ to $M$.

As $S$-modules,
$$
\Omega^{\leq m}(S)=\Omega^{[1,m]}(S)\bigoplus S\otimes 1, \qquad
\Omega^{\leq m}(R)=\Omega^{[1,m]}(R)\bigoplus R\otimes 1.
$$
Here note that $S$ acts as $S\otimes 1$.
We have
$$
\{ P\in D_R^m( R, M)\, |\, P* 1=0\}\simeq \Hom_R( \Omega^{[1,m]}(R), M)
$$
for an $R$-module $M$.

For $a\in S$ (or $R$),
we denote $1\otimes a-a\otimes 1$ mod $J_S^{m+1}$
(or $J_R^{m+1}$, respectively) by $da$.

Then, for
$f,g\in R$,
we have
\begin{equation}
\label{eqn:d(fg)}
d(fg)=f\, dg+g \, df+(df)(dg).
\end{equation}

As an $S$-module
$$
\Omega^{[1,m]}(S)=\bigoplus_{1\leq |\aalpha|\leq m}S(dx)^\aalpha.
$$
For $f\in S$, we have
\begin{equation}
\label{eqn:TalorDep}
df=\sum_{1\leq |\aalpha|\leq m}\frac{1}{\aalpha !}(\partial^\aalpha * f)(dx)^\aalpha.
\end{equation}

We have a surjective $S\otimes S$-module homomorphism
$$
\varphi: \Omega^{[1,m]}(S)\ni
(dx)^\aalpha
\mapsto (d\bar{x})^\aalpha\in \Omega^{[1,m]}(R).
$$

\begin{lemma}
As an $S$-module,
$$
\Ker\, \varphi
=
\sum_{i;\, 1\leq |\aalpha|\leq m}Sf_i(dx)^\aalpha
+
\sum_{i;\, 0\leq |\aalpha|\leq m-1}S(df_i)(dx)^\aalpha.
$$
\end{lemma}

\begin{proof}
The inclusion `$\supset$' is clear.
We prove the other inclusion.

First we prove that 
\begin{equation}
\label{eqn:11-1}
\Ker \varphi= I\, dS +S\, dI.
\end{equation}
Clearly the kernel of the $S\otimes S$-module homomorphism
:
$$\Omega^{\leq m}(S)\ni f\otimes g \mapsto \bar{f}\otimes \bar{g}\in
\Omega^{\leq m}(R)$$ 
equals $(S\otimes I+ I\otimes S)/J_S^{m+1}$
or
$(S\, dI+ I\otimes S)/J_S^{m+1}$.
Hence, to prove \eqref{eqn:11-1}, it is enough to show that
\begin{equation}
\label{eqn:11-1-2}
(I\otimes S)\cap J_S = I\, dS.
\end{equation}
Let $\sum_k i_k\otimes g_k\in J_S$ with $i_k\in I, g_k\in S$.
Then
$\sum_k i_k g_k=0$.
We have
$$
\sum_k i_k\otimes g_k=
\sum_k(i_k\otimes g_k-i_kg_k\otimes 1)
+
\sum_k i_kg_k\otimes 1=
\sum_k i_k\, dg_k +0 \in I\, dS.
$$
Hence we have proved \eqref{eqn:11-1-2} and in turn \eqref{eqn:11-1}.
Thus as an $S$-module
$$
\Ker\, \varphi
= \sum_{1\leq |\alpha|\leq m}I(dx)^\aalpha+
\sum_{0\leq |\aalpha|<m} S dI(dx)^{\aalpha}.
$$

To finish the proof, we only need to show that
$d(f_i x^\aalpha)$ belongs to the right hand of the assertion
for any $\aalpha$.
This is done
by \eqref{eqn:d(fg)}:
$$
d(f_ix^\aalpha)
=
f_id(x^\aalpha)+x^\aalpha df_i+(df_i)(d(x^\aalpha)).
$$
\end{proof}

Hence we have an $S$-free presentation of $\Omega^{[1,m]}(R)$:
\begin{equation}
\label{Presentation1}
(\bigoplus_{i;\, 1\leq |\aalpha|\leq m}
Sf_i (dx)^\aalpha)
\oplus
(\bigoplus_{i;\, 0\leq |\bbeta|\leq m-1}
S(df_i) (dx)^\bbeta)
\to
\Omega^{[1,m]}(S)
\to
\Omega^{[1,m]}(R)
\to
0.
\end{equation}

Now we consider the case $I=SQ$:
\begin{equation}
\label{Presentation2}
(\bigoplus_{1\leq |\aalpha|\leq m}
SQ (dx)^\aalpha)
\oplus
(\bigoplus_{0\leq |\bbeta|\leq m-1}
S(dQ) (dx)^\bbeta)
\to
\Omega^{[1,m]}(S)
\to
\Omega^{[1,m]}(S/SQ)
\to
0.
\end{equation}

Hence, as an $S/SQ$-module, $\Omega^{[1,m]}(S/SQ)$
has a presentation:
\begin{equation}
\label{Presentation3}
\bigoplus_{0\leq |\bbeta|\leq m-1}
(S/SQ)(dQ) (dx)^\bbeta
\to
\bigoplus_{1\leq |\aalpha|\leq m}(S/SQ)(dx)^\aalpha
\to
\Omega^{[1,m]}(S/SQ)
\to
0.
\end{equation}

Note that by \eqref{eqn:TalorDep}
\begin{eqnarray*}
(dQ)(dx)^\bbeta
&=&
\sum_{|\aalpha+\bbeta|\leq m,\, \aalpha\neq\0}
\frac{1}{\aalpha !}(\partial^\aalpha * Q) (dx)^{\aalpha+\bbeta}\\
&=&
\sum_{|\ggamma|\leq m,\, \ggamma\neq\bbeta}
\frac{1}{(\ggamma-\bbeta) !}(\partial^{\ggamma-\bbeta} * Q) (dx)^{\ggamma}.
\end{eqnarray*}

Hence the $(\bbeta, \ggamma)$-component of the matrix of \eqref{Presentation3}
equals 
$\displaystyle \frac{1}{(\ggamma-\bbeta) !}(\partial^{\ggamma-\bbeta} * Q)$.

By Lemma \ref{Im(bar-delta_0)},
the $S/SQ$-module $S^{\binom{n+m-1}{m-1}}/(J_m(\cA)+Q S^{\binom{n+m-1}{m-1}})$
has a presentation:
\begin{eqnarray}
\label{Presentation4}
\bigoplus_{1\leq |\ggamma|\leq m}(S/SQ) \frac{1}{\ggamma !}\partial^\ggamma
&\overset{\bullet}{\to}&
\bigoplus_{0\leq |\bbeta|\leq m-1} (S/SQ)\e_\bbeta\\
&\to&
S^{\binom{n+m-1}{m-1}}/(J_m(\cA)+Q S^{\binom{n+m-1}{m-1}})\to 0,\nonumber
\end{eqnarray}
and the $(\ggamma, \bbeta)$-component of the matrix of 
the map $\bullet$ in
\eqref{Presentation4} (recall \eqref{ActionBullet})
equals 
$\displaystyle \frac{1}{(\ggamma-\bbeta) !}(\partial^{\ggamma-\bbeta} * Q)$.

Thus we have proved the following theorem.

\begin{theorem}
\label{thm:TransposeOfJet}
The $S/SQ$-module
$S^{\binom{n+m-1}{m-1}}/(J_m(\cA)+Q S^{\binom{n+m-1}{m-1}})$ is the transpose of
$\Omega^{[1,m]}(S/SQ)$.
\end{theorem}

\begin{corollary}
The $S/SQ$-modules
$S^{\binom{n+m-1}{m-1}}/(J_m(\cA)+Q S^{\binom{n+m-1}{m-1}})$
and
$\Omega^{[1,m]}(S/SQ)$
share the same Fitting ideals.
\end{corollary}

 
\ifx\undefined\bysame 
\newcommand{\bysame}{\leavevmode\hbox to3em{\hrulefill}\,} 
\fi

\end{document}